%% file: ellipcenter_arxiv.tex
\documentclass{article}

\usepackage{graphicx}%
\usepackage{multirow}%
\usepackage{amsmath,amssymb,amsfonts}%
\usepackage{fullpage}
\usepackage{float}
\usepackage{url}      
\usepackage{mathrsfs}%
\usepackage{xcolor}%
\usepackage{textcomp}%
\usepackage{manyfoot}%
\usepackage{booktabs}%
\usepackage{algorithm}%
\usepackage{algorithmicx}%
\usepackage{algpseudocode}%
\usepackage{listings}%
\usepackage{tikz}
\usepackage{hyperref} 
\usetikzlibrary{calc,angles,quotes}
\usepackage{caption}
\input{def}

\title{Introducing the method of ellipcenters, a new first order technique for unconstrained optimization}

\date{}

\begin{document}

\maketitle 

\begin{center}
\begin{tabular}{ccc}
\begin{tabular}{c}
Roger Behling\\
Department of Mathematics, UFSC\\
Blumenau, SC, Brazil\\
{\tt rogerbehling@gmail.com}
\end{tabular}
&
&
\begin{tabular}{c}
Ramyro Corr\^ea Aquines\\
School of Applied Mathematics, FGV\\
Praia de Botafogo, Rio de Janeiro, Brazil\\
{\tt ramyrocorrea@gmail.com}
\end{tabular}\\
&&\\
\begin{tabular}{c}
Eduarda Ferreira Zanatta\\
Department of Mathematics, UFSC\\
Blumenau, SC, Brazil\\
{\tt eduardazanatta6@gmail.com}
\end{tabular}
&
&
\begin{tabular}{c}
Vincent Guigues\\
School of Applied Mathematics, FGV\\
Praia de Botafogo, Rio de Janeiro, Brazil\\
{\tt vincent.guigues@fgv.br}
\end{tabular}
\end{tabular}
\end{center}

\par {\textbf{Abstract.}} In this paper, we introduce the Method of Ellipcenters (ME) for unconstrained minimization. At the cost of two gradients per iteration and a line search, we compute the next iterate by setting it as the center of an elliptical interpolation. The idea behind the ellipse built in each step is to emulate the original level curve of the objective function constrained to a suitable two-dimensional affine space, which is determined by the current iterate and two appropriate  gradient vectors. We present the method for general unconstrained minimization and carry out a convergence analysis for the case where the objective function is quadratic. In this context, ME enjoys linear convergence with the rate being at least as good as the linear rate of the steepest descent (gradient) method with optimal step. In our experiments, however, ME was much faster than the gradient method with optimal step size. Moreover, ME seems highly competitive in comparison to several well established algorithms, including Nesterov's accelerated gradient, Barzilai-Borwein, and conjugate gradient. The efficiency in terms of both time and number of iterations is stressed even more for ill-conditioned problems. A theoretical feature that might be a reason for this is that ME coincides with Newton for quadratic programs in two-dimensional Euclidean spaces, solving them in one single step. 
In our numerical tests, convergence in one iteration only was also observed
for much larger problem sizes.

\section{Introduction}

The aim of this paper is to present and study the Method of Ellipcenters (ME), a novel scheme for minimizing a differentiable function $f: \mathbb{R}^n \rightarrow \mathbb{R}$. ME iterates by moving from a current iterate $x^k\in\mathbb{R}^n$ to $x^{k+1}$, a point that is the center of a special ellipse $E_k$. 

Although intending to minimize $f$, we are satisfied when finding a stationary point, that is, an $x$ so that $\nabla f(x)=0$.  That said, let us assume that we are at a non-stationary iterate $x^k$. In this case, an auxiliary point 
$y^k:=x^k-t_k\nabla f(x^k)$ is computed, where $t_k>0$ is such that $f(x^k)=f(y^k)$. Having found $y^k$, we compute the gradient $\nabla f(y^k)$ and if it is non-zero and a multiple of $\nabla f(x^k)$, ME returns $x^{k+1}$ as the next iterate, a minimizer of $f$ along the semi-line starting at $x^k$ and passing through $y^k$. Otherwise, we define the two-dimensional affine space
\begin{equation}\label{defpik}
\Pi_k:=\Big\{x\in \mathbb{R}^n :x=x^k+\mbox{span}\{\nabla f(x^k),\nabla f(y^k)\}\Big\}
\end{equation}
and build the elliptical curve $E_k$ having the following three properties:
\begin{description}
   \item[ME1] $E_k$ is an ellipse contained in $\Pi_k$;
   \item[ME2] $E_k$ is orthogonal to $\nabla f(x^k)$ at $x^k$;
   \item[ME3] $E_k$ is orthogonal to $\nabla f(y^k)$ at $y^k$. 
\end{description}

Once $E_k$ is built, ME moves from $x^k$ to $x^{k+1}$, the center of the ellipse $E_k$.

In the case of quadratic functions $f$, we take as ellipse satisfying
ME1, ME2, and ME3 the set $E_k:=\Pi_k \bigcap L_k$, where $\Pi_k$ is as above and $L_k$ is the level set 
   \begin{equation}\label{deflk}
L_k:=\left\{x \in \mathbb{R}^n : f(x) = f(x^k)\right\}.
   \end{equation}
Since $L_k$ is an ellipsoid for a quadratic objective $f$ and $\Pi_k$ is two-dimensional and affine, we have that $E_k$ is an ellipse. The center $x^{k+1}$ of $E_k$ is given as the minimizer of $f$ in $\Pi_k$, and it is obtained based on a suitable simple $2\times 2$ linear system. 

\noindent\begin{minipage}{\linewidth}
 \centering
\begin{tikzpicture}[scale=1.0]

  \draw[thick] (0,0) ellipse (4 and 1.8);  
  \node at (3.5,1.2) {$\mathbb{E}_k$};
  \node at (4.5,-1.5) {$\Pi_k$};

  \coordinate (xk) at (-2.8,-1.3);    
  \coordinate (yk) at (-1.3, 1.7);    
  \coordinate (xk1) at (0,0);

  \fill (xk) circle (2pt) node[below left] {$x^k$};
  \fill (yk) circle (2pt) node[above right] {$y^k$};
  \fill (xk1) circle (2pt) node[below right] {$x^{k+1}$};

  \coordinate (gradXk) at ($(xk)!0.7!(yk)$);  
  \coordinate (gradYk) at ($(yk) + (0.2,-1.7)$); 
  
  \draw[->, thick, >=latex] (xk) -- (gradXk) node[midway, left] {\scriptsize$-\nabla f(x^k)$};
  \draw[->, thick, >=latex] (yk) -- (gradYk) node[midway, right] {\scriptsize$-\nabla f(y^k)$};

  \draw[dashed] (xk) -- (yk);

  \coordinate (tangentXk) at ($(xk)!0.3cm!90:(gradXk)$);
  \pic [draw, angle radius=0.3cm] {right angle = tangentXk--xk--gradXk};
  
  \coordinate (tangentYk) at ($(yk)!0.3cm!90:(gradYk)$);
  \pic [draw, angle radius=0.3cm] {right angle = tangentYk--yk--gradYk};

  \node at (0,-2.2) {$f(x^k) = f(y^k)$};

\end{tikzpicture}
 
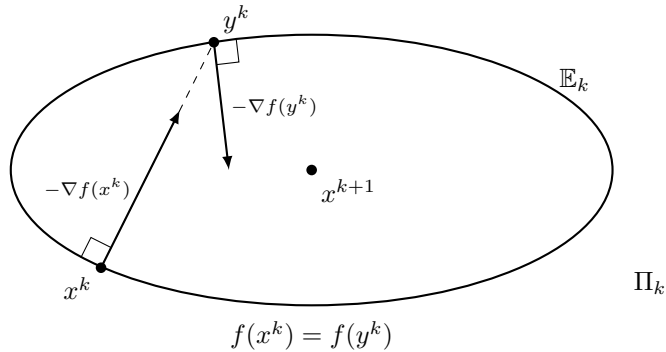
\captionof{figure}{Illustration of an iteration of ME.}\label{Grafico}
\end{minipage}\\

The idea here in principle is to pick a better point than the one given by the classical gradient method, also known as the steepest descent method or Cauchy method \cite{Cauchy1847}. More than that, the intention of $E_k$ is to emulate the level curve of $f$ restricted to $\Pi_k$ and encage a minimizer, which is slightly related to the recent
Circumcentered-Reflection-Method (CRM) proposed and studied in
for instance  \cite{Behling2024b,crm1}. Figure \ref{Grafico} illustrates an iteration of ME.

It is easy to see that even if $f$ has a global minimizer, the ME iteration might not be well-defined, there could simply not exist a point $y^k$ as required in the procedure. If, however, $f$ is strongly convex or coercive  an ellipse as described above can be constructed, and right away from ME1, ME2, and ME3 one would get that $x^{k+1}-x^k$ is a descent direction for $f$ at $x^k$. Nevertheless, there could be more ellipses $E_k$ fulfilling ME1-ME3 and, among them, a specific one would have to be chosen. Thus, in order to facilitate the understanding of ME and to clearly present its philosophy, we will assume along the remaining of the paper that $f$ is a strongly convex quadratic function. More general cases will be treated in future work. The strongly convex quadratic case alone is interesting and elegant. More than being able to build a convenient ellipse enjoying the properties ME1, ME2 and ME3, we get closed formulas for $y^k$ and can easily compute $x^{k+1}$ analytically when $f$ is strongly  convex and quadratic.

Accelerating the gradient method with first-order tools has long been a popular subject of research in the field of Continuous Optimization, see for instance \cite{grimmer2023optimal,lan2015bundle,mishchenko2020adaptive,nesterov2015universal,renegar2022simple} and \cite{zhou2024adabb}. Perhaps one of the most famous of these methods is given in the paper \cite{nesterov1983} by Y. Nesterov (fine tuning
of this algorithm is discussed in \cite{gonzagakaras}). He embedded inertia in the gradient method and was able to derive an iteration with best possible complexity for a first-order method, \textit{i.e.}, a method using first derivatives only. Another first-order algorithm that enjoys optimal complexity in theory is the Ellipsoid method \cite{Khachiyan1979}. Nevertheless, in practice, the Ellipsoid method, which is different from what we are proposing here, is not attractive.

Our paper is organized as follows. All sections deal with $f$ being a strongly convex quadratic function. Bearing this is mind, in Section \ref{sec:ellipcenter} we 
introduce ME algorithm, show that 
properties ME1, ME2 and ME3 are satisfied, and  show how the center of the ellipse $E_k$ can be computed. Along with that, we prove  that ME coincides with the Newton method in $\mathbb{R}^2$, tracking the minimizer of the quadratic in one single step. In Section \ref{secconv}, we derive linear convergence of ME and prove that this linear rate is at least as good as the one known for the gradient method with optimal step size. In Section \ref{sec:num}, we present encouraging numerical experiments, which show that ME is specially powerful for minimizing ill-conditioned quadratics even of large scale (with corresponding matrix $A$ up to size 1 million by 1 million). Indeed, ME competes with strong competitors such as Barzilai-Borwein, conjugate gradient, Nesterov's accelerated gradient and the original gradient method with optimal step (both in terms of CPU time and number of iterations). Final remarks and a to do list for future work are presented in Section \ref{secconc}.

In what follows, for vectors $x,y \in \mathbb{R}^n$, we will use the notation
$\langle x,y \rangle$ for the scalar product $x^T y$ with corresponding norm
$\|\cdot\|$. For a definite positive
matrix $A$, we will also denote by
$\langle x, y \rangle_A$ the scalar
product $x^T A y$ with corresponding norm $\|\cdot\|_A$.

\section{Method of ellipcenters for minimizing a quadratic function}\label{sec:ellipcenter}

\subsection{Algorithm}

In this section, we present and study properties of
ME applied to quadratic objective functions.
A general strongly convex quadratic can be described as a function $f: \mathbb{R}^n \rightarrow \mathbb{R}$ of the form
\begin{equation}\label{formfqd}
f(w) = \frac{1}{2} w^T A w - b^T w+c,
\end{equation}
where \( A \in \mathbb{R}^{n \times n} \) is a symmetric positive definite  matrix, \( b \in \mathbb{R}^n \), and $c$ is a scalar.

Of course $f$ has a unique minimizer $x^*= A^{-1} b\in\mathbb{R}^n$, which is also the unique point where its derivative $\nabla f(w) = A w - b$ vanishes. In other words, minimizing $f$ is equivalent to solving the linear system $Aw=b$. 

At iteration $k$, given $x^k$, if $\nabla f(x^k) \neq 0$ (otherwise we have found an optimal solution)
ME computes new iterate $x^{k+1}$ as follows. 
It computes a step $t_k$ (given below in \eqref{steptk})
and $y^{k}=x^k-t_k \nabla f(x^k)$ such that
$y^k$ and $x^k$ are in the same level set of $f$, i.e.,
$f(y^k)=f(x^k)$. 
If $\nabla f(y^k)$ and $\nabla f(x^k)$ are linearly
dependent, the next iterate is 
$x^{k+1}=(1/2)(x^k+y^k)$. Otherwise, we compute 
$\alpha_k$ and $\beta_k$ (given below in \eqref{formabk})
which are such that the next iterate
$x^{k+1}$ is given by 
$x^{k+1}=x^k + \alpha_k \nabla f(x^k) + \beta_k \nabla f(y^k)$. ME algorithm is given below.

\noindent\rule[0.5ex]{1\columnwidth}{1pt}
	Method of ellipcenters for quadratic objective function\\
    \noindent\rule[0.5ex]{1\columnwidth}{1pt}
    {\bf Inputs:} Definite positive matrix $A \in \mathbb{R}^{n \times n}$, vector $b \in  \mathbb{R}^{n}$, initial point $x^1 \in \mathbb{R}^n$, $k=1$.\\
    
\noindent {\textbf{Step 1.}} If $\nabla f(x^k)=0$ stop and return $x^k$. Otherwise, go to Step 2.\\

\noindent {\textbf{Step 2.}} Compute 
\begin{equation}\label{steptk}
t_k=\frac{2 \nabla f(x^k)^T \nabla f(x^k)}{\nabla f(x^k)^T A \nabla f(x^k)}
\end{equation}
and
\begin{equation}\label{defyk}
y^k = x^k - t_k \nabla f(x^k)
\end{equation}
where $\nabla f(x^k)=Ax^k - b$.

\noindent {\textbf{Step 3.}} If $\nabla f(y^k)$ and $\nabla f(x^k)$
are linearly independent compute
\begin{equation}\label{formabk}
\left\{
\begin{aligned}
\Delta_k &= \langle \nabla f(y^k),A \nabla f(y^k) \rangle \langle \nabla f(x^k),A \nabla f(x^k) \rangle - \langle \nabla f(x^k),A \nabla f(y^k) \rangle^2,\\ 
\alpha_k &=\displaystyle \frac{ \langle \nabla f(y^k),\nabla f(x^k) \rangle \langle \nabla f(x^k),A \nabla f(y^k) \rangle - \langle \nabla f(x^k),\nabla f(x^k)\rangle \langle \nabla f(y^k),A \nabla f(y^k)\rangle}{ \Delta_k  },\\
\beta_k &= \displaystyle \frac{ -\langle \nabla f(y^k),\nabla f(x^k) \rangle \langle \nabla f(x^k),A \nabla f(x^k) \rangle + \langle \nabla f(x^k),\nabla f(x^k)\rangle \langle \nabla f(y^k),A \nabla f(x^k)\rangle}{ \Delta_k  },
\end{aligned}
\right.
\end{equation}
$\hspace*{1.5cm}$where $\nabla f(x^k)=Ax^k - b$, $\nabla f(y^k)=Ay^k - b$, and the next iterate
\begin{equation}\label{formxkp1}
x^{k+1} = x^k + \alpha_k \nabla f(x^k) + \beta_k \nabla f(y^k);
\end{equation}
$\hspace*{1.5cm}$else if $\nabla f(y^k)$ and $\nabla f(x^k)$
are linearly dependent compute
\begin{equation}
x^{k+1} = \frac{1}{2}\left(x^k + y^k\right).
\end{equation}
$\hspace*{1.5cm}$end if\\
$\hspace*{1.5cm}$Do $k \leftarrow k+1$ and go to Step 1.\\
\noindent\rule[0.5ex]{1\columnwidth}{1pt}

In the next section, we prove properties of ME.
In particular, we show in Lemma \ref{lemxyk} that
$t_k$ allows us  to find $y^k$ such that 
$f(y^k)=f(x^k)$. We also show in Lemma
\ref{propemellipse} that $x^{k+1}$
given by \eqref{formxkp1}
is both the minimizer
of $f$ in affine space $\Pi_k$ and the center of the ellipse $E_k=\Pi_k \bigcap L_k$.

\subsection{Properties of the algorithm}

In this section, we prove the two lemmas
announced at the end of the previous section which provide basic and fundamental properties of ME.

\begin{lemma}\label{lemxyk} Vector $y^k$ given by \eqref{defyk} satisfies
$f(x^k)=f(y^k)$ and therefore $y^k$ and 
$x^k$ are in the same level set of $f$.
\end{lemma}
\begin{proof} Using the second order Taylor expansion for $f$ (which is exact since $f$ is quadratic), we have 
\begin{align}
f(y^k) &= f(x^k)-t_k \nabla f(x^k)^T \nabla f(x^k) + \frac{1}{2}(t^k)^2 \nabla f(x^k)^T A \nabla f(x^k) \\& =  f(x^k)+  t_k\Big(\underbrace{-\|\nabla f(x^k)\|^2 + \frac{1}{2}t_k \nabla f(x^k)^T A \nabla f(x^k)}_{0} \Big)\\&=f(x^k).
\end{align}
\end{proof}

\begin{lemma}\label{propemellipse} 
Consider iteration $k$ of ME and for that iteration, assume that  $\nabla f(y^k)$ and $\nabla f(x^k)$
are linearly independent.
Then the  ellipse
$$
E_k=\Pi_k \bigcap L_k
$$
where $\Pi_k$ is defined in \eqref{defpik} 
and $L_k$ is defined by \eqref{deflk} satisfies ME1, ME2, ME3 and $x^{k+1}$
given by \eqref{formxkp1} is both the minimizer
of $f$ in $\Pi_k$  and the center of $E_k$.
\end{lemma}
\begin{proof}
 From elementary properties of derivatives, the gradients $\nabla f(x^k)$, $\nabla f(y^k)$ have to be orthogonal to the level set $L_{k}$ at $x^k$, and $y^k$, respectively. Since the curve $E_k$ is, by definition, contained in $L_{k}$, we get that
 $ME_2$ and $ME_3$ are satisfied. 
 Now, $L_k$ is an ellipsoid and $\Pi_k$ is a two-dimensional affine set. Therefore, $E_k:=\Pi_k \bigcap L_k$ is an ellipse contained in $\Pi_k$
 and $ME_1$ is satisfied. 

 The description of this ellipse $E_k$ can be done by writing down $f(x^k + \alpha \nabla f(x^k) + \beta \nabla f(y^k))$ using two real parameters $\alpha$ and $\beta$. Points of form $x^k+\alpha \nabla f(x^k) + \beta \nabla f(y^k)$ of the level
 set $L_k$ can be written as the set
 of points
 $x^k+\alpha \nabla f(x^k) + \beta \nabla f(y^k)$
 where $(\alpha,\beta)$
 satisfies
 \begin{equation}\label{palphbeta}
 p(\alpha,\beta)=f(x^k + \alpha \nabla f(x^k) + \beta \nabla f(y^k)) -f(x^k)=0.
\end{equation}
with $p:\mathbb{R}^2\to\mathbb{R}$. 
We obtain an ellipse whose center
is $x^k+\alpha^* \nabla f(x^k) + \beta^* \nabla f(y^k)$
 where $(\alpha^*,\beta^*)$
 satisfies
$\nabla p(\alpha^*,\beta^*)=0$.
This equation can be written
\begin{align}
\langle \nabla f(x^k) , \nabla f(x^k + \alpha^* \nabla f(x^k) +  \beta^* \nabla f(y^k))  \rangle = \langle \nabla f(x^k) , A(x^k + \alpha^* \nabla f(x^k) +  \beta^* \nabla f(y^k))-b  \rangle  = 0,\label{systabeta1}\\
\langle \nabla f(y^k) , \nabla f(x^k + \alpha^* \nabla f(x^k) +  \beta^* \nabla f(y^k))  \rangle = \langle \nabla f(y^k) , A(x^k + \alpha^* \nabla f(x^k) +  \beta^* \nabla f(y^k))-b  \rangle  = 0.\label{systabeta2}
\end{align}
The system above is of form
$$
M^k \left[
\begin{array}{c}
\alpha^*\\
\beta^*
\end{array}
\right] = q^k 
$$
where 
\begin{equation}\label{formMk}
M^k = \left[ 
\begin{array}{cc}
\langle \nabla f(x^k), \nabla f(x^k) \rangle_A & \langle \nabla f(x^k), \nabla f(y^k) \rangle_A \\
\langle \nabla f(x^k), \nabla f(y^k) \rangle_A & \langle \nabla f(y^k), \nabla f(y^k) \rangle_A
\end{array}
\right],\;\;\;q^k=\left[ 
\begin{array}{cc}
-\|\nabla f(x^k)\|^2 \\
-\langle \nabla f(x^k), \nabla f(y^k) \rangle
\end{array}
\right].
\end{equation}
Matrix $M^k$ is a Gram matrix and since
we are in the situation where
vectors $\nabla f(x^k)$ and 
$\nabla f(y^k)$ are linearly independent, this matrix is invertible 
($A$ is positive definite) and system of equations \eqref{systabeta1}-\eqref{systabeta2}
in 
variables $(\alpha^*,\beta^*)$
has a unique solution.
Solving for $\alpha^*$, $\beta^*$, we 
find
 $\alpha^*=\alpha_k$,
$\beta^*=\beta_k$, which shows that
the center of the ellipse is indeed
$x^k+\alpha_k \nabla f(x^k) + \beta_k \nabla f(y^k)$. From the above computations, it is also clear that
$x^k+\alpha_k \nabla f(x^k) + \beta_k \nabla f(y^k)$ is a minimizer of
$f$ in $\Pi_k$, i.e., is of form
$x^k+\alpha^* \nabla f(x^k) + \beta^* \nabla f(y^k)$ with $(\alpha^*,\beta^*)$
solution of \eqref{systabeta1}-\eqref{systabeta2}.
\end{proof}

\par Observe that the denominator
$\Delta_k$ in the formulas 
giving $\alpha_k$ and
$\beta_k$ is positive ($\Delta_k$ is the determinant of $M^k$ given by \eqref{formMk}).
Indeed, $\Delta_k=\| \nabla f(y^k)\|^2_A
\|\nabla f(x^k)\|^2_A -
\langle \nabla f(x^k), \nabla f(y^k) \rangle_A^2$ and by Cauchy-Schwarz inequality, $\Delta_k \geq 0$
and $\Delta_k = 0$ (there is equality in Cauchy-Schwartz inequality) if and only if
$\nabla f(y^k)$ and $\nabla f(x^k)$
are linearly dependent.
Since when computing $\Delta_k$ in \eqref{formabk} we have that 
$\nabla f(y^k)$ and $\nabla f(x^k)$
are linearly independent, we indeed have
that $\Delta_k>0$.\\

If $\nabla f(x^k)\neq 0$ and, furthermore, $\nabla f(y^k)$ is a multiple of $\nabla f(x^k)$, then the next iterate $x^{k+1}$ of ME  is $x^{k+1}=1/2(x^k+y^k)=x^k-t_k^* \nabla f(x^k)$
where $t_k^*=t_k/2$ is the step taken from $x^k$
by the gradient method with optimal step.
Therefore, in this situation, as intended, 
the next iterate $x^{k+1}$ is a 
minimizer of $f$ along the semi-line starting at $x^k$ in the direction of $y^k$. This means that
when $\nabla f(x^k)$ and $\nabla f(y^k)$ are linearly dependent, the ME iteration coincides with an iteration of the gradient method with optimal step. This is consistent with the philosophy of ME, as one could get an ellipse with the properties  ME2 and ME3 by simply taking any circle centered in $x_{k+1}$ passing through $x^k$ and $y^k$. The common situation, however is the one considered in Lemma \ref{propemellipse}, where $\nabla f(x^k)\neq 0$ and $\nabla f(y^k)$ are linearly independent, in which case, the next iterate is given by \eqref{formxkp1}.\\

In the next section, we prove the linear convergence of ME applied to quadratic objective functions.

\section{Convergence analysis} \label{secconv}

To prove the linear convergence of
ME, we start recalling Kantorovich
inequality, see \cite{Newman1959}.

\begin{proposition}[Kantorovich inequality] Let $A$
be a symmetric definite positive matrix with eigenvalues
$0<\lambda_1 \leq \lambda_2 \leq \ldots \leq \lambda_n$.
Then for any vector $y \neq 0$, we have
\begin{equation}
\frac{(y^T y)^2}{(y^T A y)  (y^T A^{-1} y)} \geq 
\frac{4 \lambda_1 \lambda_n}{(\lambda_1+\lambda_n)^2}.
\end{equation}
\end{proposition}

The linear convergence of ME is given in the next theorem.

\begin{theorem}
 Let $x^1\in \mathbb{R}^n$ be any initial point for ME. 
 Let 
 $0<\lambda_1 \leq \lambda_2 \leq \ldots \leq \lambda_n$
 be the eigenvalues of $A$.
 Then, ME generates a sequence $\{x^k\}\subset \mathbb{R}^n$ that converges to $x^*$, the unique minimizer of the quadratic $f$. Moreover, the convergence is linear and we have $\eta \leq 1 - \frac{\lambda_1}{\lambda_n} \in [0,1)$ such that for all $k \geq 1$,
 \begin{equation}\label{firstem}
    f(x^{k+1}) - f(x^*) \leq \eta^k (f(x^1) - f(x^*))
\end{equation}
and
\begin{equation}\label{secem}
    \|x^{k+1} - x^*\|_A \leq \sqrt{\eta}^k\|x^1 - x^*\|_A. 
\end{equation}
\end{theorem}
\begin{proof} If for some $k$ we have $\nabla f(x^k)=0$ then $x^k=x^*$
is an optimal solution and the algorithm converges in a finite number of iterations.
From now on, we assume that for all $k$ we have $\nabla f(x^k) \neq 0$.
Since $x^{k+1}$ minimizes
$f$ on the affine space $\Pi_k$, we have 
$$
f(x^{k+1}) \leq f(x),\;\;\forall x \in \Pi_k.
$$
In particular, for every $t>0$ we have
$$
f(x^{k+1}) \leq f(x^k-t\nabla f(x^k)),
$$
which implies, for $t=t_k^*=t_k/2$ that
\begin{align} \label{eqpr1}
f(x^{k+1}) -f(x^*) &\leq f(x^k-t_k^* \nabla f(x^k))-f(x^*) \nonumber\\
&= f(x^k)-t_k^* \|\nabla f(x^k)\|^2 + \frac{(t_k^*)^2}{2}\nabla f(x^k)^T A \nabla f(x^k)-f(x^*) \nonumber\\
& = f(x^k) -f(x^*) - \frac{1}{2} \frac{\| \nabla f(x^k) \|^4}{\nabla f(x^k)^T A \nabla f(x^k)} \;\;\mbox{using the expression of }t_k^*.
\end{align}
Using Kantorovich inequality with $y = \nabla f(x^k)\neq 0$, we obtain
\begin{equation}\label{prem2}
\frac{\| \nabla f(x^k) \|^4}{\nabla f(x^k)^T A \nabla f(x^k)} \geq \frac{4\lambda_1 \lambda_n}{(\lambda_1 + \lambda_n)^2} \cdot \nabla f(x^k)^T A^{-1} \nabla f(x^k).
\end{equation}
Since $\nabla f(x^*)=Ax^*-b=0$, we have
that $\nabla f(x^k) = Ax^k -b=A(x^k - x^*)$, which gives
\begin{equation}\label{prem3}
\nabla f(x^k)^T A^{-1} \nabla f(x^k) = (x^k - x^*)^T A (x^k - x^*) = 2(f(x^k) - f(x^*)).
\end{equation}
Combining \eqref{eqpr1},
\eqref{prem2}, and \eqref{prem3}, we have
\begin{align}
f(x^{k+1}) - f(x^*) &\leq f(x^k) - f(x^*) - \frac{4\lambda_1 \lambda_n}{(\lambda_1 + \lambda_n)^2} (f(x^k) - f(x^*)) \nn\\
&= \left( \frac{\lambda_n - \lambda_1}{\lambda_n + \lambda_1} \right)^2 (f(x^k) - f(x^*))\nn \\
&\leq \left(1 - \frac{\lambda_1}{\lambda_n}  \right) (f(x^k) - f(x^*)),
\end{align}
which achieves the proof of \eqref{firstem}.

To show \eqref{secem}, we use the identity
\begin{align}
f(x) - f(x^*) = \frac{1}{2}(x-x^*)^\top A(x-x^*) = \frac{1}{2}\|x-x^*\|_A^2,
\end{align}
which gives 
\[
\|x^{k+1} - x^*\|_A^2 = 2  \big( f(x^{k+1}) - f(x^*) \big)
\;\leq\; 2 \left(1 - \frac{\lambda_1}{\lambda_n}  \right) \big( f(x^k) - f(x^*) \big)
= \left(1 - \frac{\lambda_1}{\lambda_n}  \right) \|x^k - x^*\|_A^2
\]
and achieves the proof of the theorem.
\end{proof}

Finally, we argue that the linear rate of convergence of ME is at least as good (and in practice form our experiments of
Section \ref{sec:num} much better) as the rate of convergence of the gradient method
with optimal step.

\begin{proposition}
Let $x \in \mathbb{R}^n$ with $\nabla f(x) \neq 0$ and consider an
iteration of ME that computes the next iterate $x_{ME}$ from $x$
and an iteration of the gradient method with optimal step size that
computes the next iterate $x_{\mbox{grad}}$ from $x$.
Then $f(x_{ME}) \leq f(x_{\mbox{grad}})$.
\end{proposition}
\begin{proof}
Let $$
y=x-t \nabla f(x) \mbox{ for }  t=\frac{2 \nabla f(x)^T \nabla f(x)}{\nabla f(x)^T A \nabla f(x)}.
$$
If $\nabla f(x)$ and $\nabla f(y)$ are linearly dependent,
then $x_{\mbox{grad}}=x_{ME}$ and 
$f(x_{ME}) \leq f(x_{\mbox{grad}})$.
Otherwise, $x_{ME}$ minimizes $f$ in $\Pi=\{x+\alpha \nabla f(x)+\beta \nabla f(y): (\alpha,\beta) \in \mathbb{R}^2\}$ which gives
$$
f(x_{ME}) \leq f(x + \alpha \nabla f(x) + \beta \nabla f(y))
$$
for all $\alpha$, $\beta$. In particular,
$$
f(x_{ME}) \leq f(x - (t/2) \nabla f(x) )=f(x_{\mbox{grad}})
$$
where $t^*=t/2$ is the step taken from $x$ by the gradient method
with optimal step. 
This achieves the proof.
\end{proof}

\section{Numerical experiments}\label{sec:num}

In this section, we consider several instances of the problem of minimizing a quadratic function of form \eqref{formfqd} with $A$ definite positive. 

We start taking $A$ diagonal with
a large condition number equal to 50000, which is a situation where the gradient method with optimal step is known to be slow. Of course, for diagonal matrices, we have an analytic solution which can be immediately computed but all methods are run without having the information that $A$ is diagonal. We generate $A \in \mathbb{R}^{n \times n}$ diagonal as follows:
the first entry in the diagonal is 1, the last entry in the diagonal is 50000 while the remaining
entries in the diagonal are generated sampling integers between 10 and
49900 from the uniform distribution on the corresponding set
of integers. The following (large) values
of the problem size are selected: $n=100000$,
$n=150000$, $n=200000$, $n=250000$, $n=500000$,
$n=700000$, $n=850000$, and $n=1000000$.
We compare ME with the following methods: gradient with optimal step, fast gradient \cite{nesterov1983}, Barzilai-Borwein with long steps \cite{BarzilaiBorwein1988}, 
Barzilai-Borwein with short steps \cite{BarzilaiBorwein1988}, and conjugate gradient \cite{HestenesStiefel1952}  
\cite{FletcherReeves1964}. BFGS method  (see \cite{Fletcherroger87}) was implemented but not tested since it could not run for these large size problems
with the RAM at our disposal. Indeed, the method requires storing huge (for our experiments) 
approximate inverse of Hessian matrices of size $n \times n$. 
Gradient method with Wolfe line search \cite{Wolfe1969} was also implemented but the method
was very slow and the corresponding results are not reported.
Algorithms are stopped at an iteration $k$ when $\|\nabla f(x^k)\| \leq \varepsilon$.
The methods were implemented in Julia and the corresponding code is available on github at {\url{https://github.com/vguigues/Ellipcenter-method}}.
The methods were run on a laptop
with processor Intel Core i7-12700H, 2.3GHz and 16GB of RAM.

Before presenting the results, we now provide, for the sake of completeness, the pseudo-code of all
tested optimization methods.

We start with the pseudo-code of gradient method with optimal step.
At every iteration, given gradient $r=Ax-b$ of the objective computed at the current iterate $x$, the optimal step
is computed as 
$t=\|r\|^2/r^T A r$ and the new iterate $x-tr$ is then computed.\\
{ \small{
\noindent\rule[0.5ex]{1\columnwidth}{1pt}
Gradient method with optimal step to minimize $f$ given by \eqref{formfqd}\\
    \noindent\rule[0.5ex]{1\columnwidth}{1pt}
    {\bf Inputs:} Initial point $x \in \mathbb{R}^n$, $k=0$, $r = Ax - b$.\\
{\textbf{While}} $\|r\|>\varepsilon$\\
$\hspace*{1cm}$ $t = \frac{\|r\|^2}{r^T A r}$\\
$\hspace*{1cm}$ $x \leftarrow x-tr$\\
$\hspace*{1cm}$ $r = Ax - b$\\
$\hspace*{1cm}$ $k \leftarrow k+1$\\
{\textbf{End While}}\\
\noindent\rule[0.5ex]{1\columnwidth}{1pt}
}}
We now provide the pseudo-code
of conjugate gradient which computes
at iteration $k$ a step $t_k$
and a descent direction $d^k$
where descent directions
$d^1,d^2,\ldots,d^k$ are conjugate
with respect to $A$.

{ \small{
\noindent\rule[0.5ex]{1\columnwidth}{1pt}
Conjugate gradient to minimize $f$ given by \eqref{formfqd}\\
    \noindent\rule[0.5ex]{1\columnwidth}{1pt}
    {\bf Inputs:} Initial point $x^0 \in \mathbb{R}^n$, $k=0$.\\
{\textbf{While }}$\nabla f( x^0 ) \neq 0$ {\textbf{do}}\\
\hspace*{1.15cm}{\textbf{If}} $k=0$ {\textbf{then}}\\ 
\hspace*{2.8cm}$d^0=\nabla f (x^0 )$\\ 
\hspace*{1.2cm}{\textbf{else}} 
$$
d^k = \nabla f ( x^k ) + \theta_{k-1} d^{k-1} \mbox{ where }\theta_{k-1} = -\frac{\langle \nabla f( x^{k} ) , A d^{k-1} \rangle}{\langle  d^{k-1} , A d^{k-1} \rangle}. 
$$
\hspace*{1.2cm}{\textbf{End if}}\\
\hspace*{1.2cm}Compute 
$$
t_k = - \frac{\langle d^{k} , \nabla f( x^{k} ) \rangle}{\langle d^{k} , A d^{k} \rangle} \mbox{ and }x^{k+1} = x^k + t_k d^k.
$$
\hspace*{1.2cm}Do $k \leftarrow k + 1$.\\
{\textbf{End While}}\\
\noindent\rule[0.5ex]{1\columnwidth}{1pt}
}}

To compute the step in the first iteration, Barzilai-Borwein uses a line search.
In our implementation, Wolfe search \cite{Wolfe1969} was used. Given
as inputs
a current iterate $x$, a descent direction $d$, a function {\em{blackbox}} (which given
an input $x$ outputs the value and gradient of the objective function at $x$),
an extrapolation parameter $a>1$, and two parameters $0<m_1<m_2<1$, this function
Wolfe\_search outputs a step to take in descent direction $d$ from $x$.

\if{
which will be used to compute the first step in Barzilai-Borwein method.
Wolfe line search updates along iterations an interval $[t_L,t_R]$
which contains satisfactory steps.
The function takes as inputs a given iterate $x$, a descent direction $d$, three parameters: $a>1$ (for extrapolations) and $0<m_1<m_2<1$, and a function called
blackbox which, given as input $x$, returns two outputs, the first being the value of the objective function at $x$ and the second being the derivative of the function at $x$.
The Wolfe line search function 
returns a satisfactory step (allowing a sufficient decrease in the objective and  increase in the
derivative) $t$ satisfying
in particular
$f(x+td)<f(x)$.

{\small{
\noindent\rule[0.5ex]{1\columnwidth}{1pt}
Wolfe\_search\\
    \noindent\rule[0.5ex]{1\columnwidth}{1pt}
    {\bf Inputs:} Initial point $x \in \mathbb{R}^n$,
    descent direction $d$, function {\em{blackbox}}, parameters a, $m_1$, $m_2$\\
$t = 1$\\
$t_{L} = 0$\\
$t_{R} = \infty$\\
bool = false\\
{\bf While } !bool\\
\hspace*{1.2cm}[$f_1$,$g_1$]=blackbox(x+td)\\
\hspace*{1.2cm}[$f_2$,$g_2$]=blackbox(x)\\
\hspace*{1.2cm}qt = $f_1$\\
\hspace*{1.2cm}qpt = $d^T g_1$\\
\hspace*{1.2cm}q0 = $f_2$\\
\hspace*{1.2cm}qp0 = $d^T g_2$\\
\hspace*{1.2cm}if (qt $\leq$ q0 + m1 qp0 t) and (qpt $\geq$ m2 qp0)\\
\hspace*{1.6cm}bool = true\\
\hspace*{1.2cm}elseif (qt $\leq$ q0 + m1 qp0 t) and  (qpt $<$  m2 qp0)\\
\hspace*{1.6cm}$t_L = t$\\
\hspace*{1.2cm}else\\
\hspace*{1.6cm}$t_R = t$\\
\hspace*{1.2cm}end\\
\hspace*{1.2cm}if $t_R = \infty$\\
\hspace*{1.6cm}$t \leftarrow a t$\\
\hspace*{1.2cm}else\\
\hspace*{1.6cm}$t=\frac{t_L+t_R}{2}$\\
\hspace*{1.2cm}end\\\
{\bf{End While}}\\
{\bf Output:} $t$\\
\noindent\rule[0.5ex]{1\columnwidth}{1pt}
}}

}\fi

The following pseudo-code is Barzilai-Borwein method with Wolfe search to compute the step at the first iteration. It uses a boolean short\_steps which is 1 (resp. 0) if short steps (resp. long steps) are used.

{ \small{
\noindent\rule[0.5ex]{1\columnwidth}{1pt}
Barzilai-Borwein algorithm\\
    \noindent\rule[0.5ex]{1\columnwidth}{1pt}
    {\bf Inputs:} Initial point $x \in \mathbb{R}^n$, $k=0$, $d=b-Ax$, short\_steps: boolean which is one if short steps are chosen and 0 otherwise (in which case long steps are chosen).\\
xprev=0\\
dprev=0\\
{\bf While} $\|d\|>\varepsilon$\\
\hspace*{1.2cm}if $k=0$\\
\hspace*{1.6cm}step = Wolfe\_search(x,d,blackbox, a,$m_1$,$m_2$)\\
\hspace*{1.2cm}else\\
\hspace*{1.6cm}$s=x-\mbox{xprev}$\\
\hspace*{1.6cm}$y=\mbox{dprev}-d$\\
\hspace*{1.6cm}if short\_steps //short steps\\\
\hspace*{2cm}$t = \frac{s^T y}{y^T y}$\\
\hspace*{1.6cm}else //long steps\\
\hspace*{2cm}$t = \frac{s^T s}{s^T y}$\\
\hspace*{1.6cm}end\\
\hspace*{1.2cm}end\\
\hspace*{1.2cm}xprev=x\\
\hspace*{1.2cm}dprev= d\\
\hspace*{1.2cm}$x \leftarrow x+t d$\\
\hspace*{1.2cm}d = b - Ax\\
\hspace*{1.2cm}k$\leftarrow k+1$\\
{\bf End While}\\
\noindent\rule[0.5ex]{1\columnwidth}{1pt}
}}

Finally, we give the pseudo-code of
the fast-gradient method from \cite{nesterov1983} to minimize \eqref{formfqd}. It depends on a parameter $L$ which satisfies
$\|\nabla f(y)-\nabla f(x)\|_2 \leq L \|y-x\|$ for all $x,y$. We take 
$L=\|A\|_2$ for this parameter.

{\small{
\noindent\rule[0.5ex]{1\columnwidth}{1pt}
Fast-gradient \cite{nesterov1983} to minimize $f$ given by \eqref{formfqd}\\
    \noindent\rule[0.5ex]{1\columnwidth}{1pt}
    {\bf Inputs:} Initial point $x \in \mathbb{R}^n$, $k=0$, $d=b-Ax$.\\    
$y = x$, $C = 0$, $k=0$, $L=\|A\|_2$\\
{\bf While} $\|Ax-b\|>\varepsilon$\\
\hspace*{1.2cm}$a = \frac{1}{2L} \left( 1+\sqrt{1+4 L C} \right)$\\
\hspace*{1.2cm}$C^{+} = C + a$\\
\hspace*{1.2cm}$\tilde x =\frac{1}{C^+}\Big(C y + a x\Big)$\\  
\hspace*{1.2cm}$y^{+} = \tilde x + \frac{1}{L}(b-A \tilde x)$\\ 
\hspace*{1.2cm}$x=\frac{C^{+}}{a}y^{+} - \frac{C}{a}y$\\  
\hspace*{1.2cm}$y=y^{+}$\\
\hspace*{1.2cm}$C=C^{+}$ \\ 
\hspace*{1.2cm}$k \leftarrow k+1$\\
{\bf End While}\\
\noindent\rule[0.5ex]{1\columnwidth}{1pt}   
}}

\if{
Finally, we provide the pseudo-code of the gradient method with Wolfe line search. The method has the same parameters $a, m_1, m_2$ and function {\em{blackbox}} 
as Wolfe\_search method.

{\small{
\noindent\rule[0.5ex]{1\columnwidth}{1pt}
Gradient method with Wolfe line search\\
    \noindent\rule[0.5ex]{1\columnwidth}{1pt}
{\bf{Inputs:}} initial point $x \in \mathbb{R}^n$, parameters $a>1$, 
$0<m_1<m_2<1$,  function blackbox (same as in Wolfe search), $k=0$, $k_{\max}$.\\
$[f,g]$=blackbox(x)\\
{\bf While } $\|g\|>\varepsilon$ and $k < k_{\max}$\\
\hspace*{1.4cm}t=Wolfe\_search(x,-g,blackbox,a,$m_1$,$m_2$)\\
\hspace*{1.4cm}$x \leftarrow x-tg$\\
\hspace*{1.4cm}$[f,g]$=blackbox(x)\\
\hspace*{1.4cm}k$\leftarrow k+1$\\
{\bf End While}\\
\noindent\rule[0.5ex]{1\columnwidth}{1pt}   
}}

}\fi

\par {\textbf{Results.}} We report in Tables \ref{tableres1}  and \ref{tableres2} for the instances with diagonal matrix $A$ the condition number
$\lambda_{\max}(A)/\lambda_{\min}(A)$, the number of iterations, the CPU time (in seconds), and the optimal value found at termination. All methods find the same approximate optimal value on all problem instances which is an indication that all methods were correctly implemented.
The three (by far) quickest methods are
conjugate gradient, Barzilai-Borwein, and the method of ellipcenters which provide close CPU times. The gradient method with optimal step  and the fast-gradient method were
much slower and required many more iterations.
In terms of number of iterations, the 
conjugate gradient, Barzilai-Borwein, and the method of ellipcenters are also the methods requiring the smallest number of iterations
with the method of ellipcenters ranked second
and providing a number of iterations close to the number of iterations obtained with (the fastest) conjugate gradient. Since conjugate gradient is known to be the most efficient to date for solving linear systems $Ax=b$ with positive
definite $A$, these preliminary results are very encouraging for ME and its extension to more general problems.

We then compared the same methods considering dense
matrices $A \in \mathbb{R}^{n \times n}$ for
$n=40, 50,$ $100, 150, 200, 300$, $400,$ $500, 600, 700,$
and $1000$.
These matrices are generated of the form
$A=v v^T + 10I_n$ where entries of vector $v$
are generated taking a sample from the uniform distribution on
the interval $[0,1]$. 
For these experiments, we limited the
number of iterations of the fast-gradient method to 1000.
The condition number,  CPU time, number of iterations, and optimal value with these methods is reported
in Tables \ref{tableres3}  and \ref{tableres4}.
On the 11 instances of this experiment, the number of iterations of the method of ellipcenters
and conjugate gradient was remarably low and these methods always find an approximate optimal solution in 1 or 2 iterations with the method of ellipcenters requiring less iterations (only one) than conjugate gradient in 5 of the experiments, the same number of iterations
(2 iterations) in 4 of the experiments, and conjugate gradient requiring less iteration than the method of ellipcenters in only 2 of the experiments.
The method of ellipcenters and conjugate gradient are the quickest in
these instances.

{\tiny{
\begin{table}
\centering
\begin{tabular}{|c|c|c|c|c|c|}
 \hline
Method & n &  $\displaystyle \frac{\lambda{\max}(A)}{\lambda_{\min}(A)}$  &   CPU time (s) & Iterations & Optimal value \\
 \hline
 ME &100000&  50000 & 0.54  & 21  &		-1.524e6\\
 \hline
 Gradient optimal step  &100000 &50000 & 5.49& 2929   	& -1.524e6\\
 \hline
Fast gradient (Nesterov)  &100000& 50000 &96.2 & 31803 		& -1.524e6\\
 \hline
Barzilai-Borwein long steps  &100000& 50000 &0.32  & 35 &	-1.524e6 	\\
 \hline
Barzilai-Borwein short steps  &100000& 50000 &0.27 &25 	& -1.524e6\\
 \hline
 Conjugate gradient  &100000&50000 &0.1 & 18 	& -1.524e6\\
 \hline
 \hline
ME &150000& 50000 & 1.77   & 23  &	-2.275e6\\
 \hline
 Gradient optimal step  &150000 &50000 &30.3 & 4761   	& -2.275e6\\
 \hline
Fast gradient (Nesterov)  &150000& 50000 &2747.2& 	32192	&-2.275e6 \\
 \hline
Barzilai-Borwein long steps  &150000&50000 &0.57 &34  &	-2.275e6 	\\
 \hline
Barzilai-Borwein short steps  &150000&50000 & 0.67&29  	& -2.275e6\\
 \hline
 Conjugate gradient  &150000& 50000 &0.32& 18 	& -2.275e6\\
 \hline
 \hline
 ME &200000 & 50000    & 2.22 &25  &	-3.029e6\\
 \hline
 Gradient optimal step  &200000 & 50000 &1454.0 & 5223   	& -3.029e6\\
 \hline
Fast gradient (Nesterov)  &200000& 50000 &784.0 & 33461		& -3.029e6\\
 \hline
Barzilai-Borwein long steps  &200000& 50000 &0.79&  37 &	-3.029e6 	\\
 \hline
Barzilai-Borwein short steps  &200000& 50000 &0.57 & 27 	& -3.029e6\\
 \hline
 Conjugate gradient  &200000 &50000 & 0.40& 19 	& -3.029e6 \\
 \hline 
 \hline
 ME & 250000 &50000 & 2.02 & 23  &	-3.772e6\\
 \hline
 Gradient optimal step  & 250000& 50000 &74.9  &  8701  & -3.772e6\\
 \hline
Fast gradient (Nesterov)  &250000& 50000 &1089.4  & 35919		& -3.772e6\\
 \hline
Barzilai-Borwein long steps  &250000& 50000 &0.81& 34 &	 -3.772e6	\\
 \hline
Barzilai-Borwein short steps  &250000&50000 & 0.63&  27	& -3.772e6\\
 \hline
 Conjugate gradient  &250000&50000 & 0.52& 19  	& -3.772e6\\
 \hline
 \end{tabular}
 \vspace{0.5cm}
\caption{Comparison on several instances with diagonal matrix $A$ (for several values of problem size $n$) of the number of iterations, CPU time (in seconds), and optimal value at termination for ME,
gradient with optimal step,  fast gradient (Nesterov), Barzilai-Borwein with long steps, 
Barzilai-Borwein with short steps, and conjugate gradient methods using $\varepsilon=1$.}\label{tableres1}
\end{table}
}}

{\tiny{
\begin{table}
\centering
\begin{tabular}{|c|c|c|c|c|c|}
 \hline
Method & n & $\displaystyle \frac{\lambda{\max}(A)}{\lambda_{\min}(A)}$  &  CPU time (s) & Iterations & Optimal value \\
\hline
ME & 500000 & 50000 & 2.94   & 21  &	-7.526e6\\
 \hline
 Gradient optimal step  & 500000& 50000 &132.8 &  8181 & -7.526e6\\
 \hline
Fast gradient (Nesterov)  &500000& 50000 & 16444  & 	38358	& -7.526e6\\
 \hline
Barzilai-Borwein long steps  &500000& 50000 & 1.26 &  31 &	 -7.526e6	\\
 \hline
Barzilai-Borwein short steps  &500000& 50000 &1.31 & 26 	& -7.526e6\\
 \hline
 Conjugate gradient  &500000& 50000 &0.72 & 19 	& -7.526e6\\
 \hline
 \hline
ME &700000& 50000 & 4.68   & 23  &	-1.052e7\\
 \hline
 Gradient optimal step  &700000 &50000 &490.3 & 19403   	& -1.052e7\\
 \hline
Fast gradient (Nesterov)  &700000&50000 & 3254.9 & 	43506	& -1.052e7\\
 \hline
Barzilai-Borwein long steps  &700000&50000 &1.77 &33  &	-1.052e7 	\\
 \hline
Barzilai-Borwein short steps  &700000&50000 & 1.82& 31 	& -1.052e7\\
 \hline
 Conjugate gradient &700000 &50000 &0.89& 19   	& -1.052e7\\
 \hline
 \hline
 ME &850000& 50000 &  4.78  & 25  &	-1.277e7\\
 \hline
 Gradient optimal step  &850000 &50000 & 587.9 & 18025    	& -1.277e7\\
 \hline
Fast gradient (Nesterov)  &850000&50000 & 7928,8  & 44256		& -1.277e7\\
 \hline
Barzilai-Borwein long steps  &850000&50000 &2.30 & 35 &-1.277e7\\
 \hline
Barzilai-Borwein short steps  &850000& 50000 &2.21  & 30 	& -1.277e7\\
 \hline
 Conjugate gradient  &850000& 50000 &1.03& 19 	& -1.277e7\\
 \hline
 \hline
 ME &1000000& 50000 & 6.05   & 25  & -1.503e7	\\
 \hline
 Gradient optimal step  & 1000000& 50000 & 575.4  & 15557   	& -1.503e7\\
 \hline
Fast gradient (Nesterov)  &1000000&50000 & 4744.9 & 	44997	& -1.503e7\\
 \hline
Barzilai-Borwein long steps  &1000000& 50000 &2.37& 33 &	 -1.503e7	\\
 \hline
Barzilai-Borwein short steps  &1000000& 50000 &2.36 & 30  	& -1.503e7\\
 \hline
 Conjugate gradient  &1000000&50000 &1.22 & 19 	& -1.503e7\\
 \hline
 \end{tabular}
 \vspace{0.5cm}
\caption{Comparison on several instances with diagonal matrix $A$ of the number of iterations, CPU time, and optimal value at termination for ME,
gradient with optimal step, fast gradient, Barzilai-Borwein with long steps, 
Barzilai-Borwein with short steps, and conjugate gradient methods using $\varepsilon=1$.}\label{tableres2}
\end{table}
}}

{\tiny{
\begin{table}
\centering
\begin{tabular}{|c|c|c|c|c|c|}
 \hline
Method & n & $\displaystyle \frac{\lambda{\max}(A)}{\lambda_{\min}(A)}$ &CPU time (s) & Iterations & Optimal value \\
 \hline
 ME &40&  20   & $<10^{-16}$ &	1&-200.5\\
 \hline
 Gradient optimal step  &40 & 20 & $<10^{-16}$ & 15  	& -200.5\\
 \hline
Fast gradient (Nesterov)  &40& 20  &0.002 	&	466& -200.5\\
 \hline
Barzilai-Borwein long steps  &40&20 & $<10^{-16}$ & 9 &-200.5\\
 \hline
Barzilai-Borwein short steps   &40&20   & $<10^{-16}$&7 	&-200.5 \\
 \hline
 Conjugate gradient  &40& 20& $<10^{-16}$ 	& 2&-200.5\\
 \hline
 \hline
ME &50& 23    & $<10^{-16}$ &1	&-267.0\\
 \hline
 Gradient optimal step  &50 &23  & 0.0009 &15   	&-267.0 \\
 \hline
Fast gradient (Nesterov)  &50&23   & 0.003	&668	& -267.0\\
 \hline
Barzilai-Borwein long steps  &50&23 & $<10^{-16}$ & 9&-267.0\\
 \hline
Barzilai-Borwein short steps   &50&  23 & $<10^{-16}$& 7	& -267.0\\
 \hline
 Conjugate gradient  &50&23 & $<10^{-16}$ 	& 2&-267.0\\
 \hline
 \hline
ME &100&  54   & $<10^{-16}$  &	1&-1454.5\\
 \hline
 Gradient optimal step  &100 & 54 & $<10^{-16}$ & 11  	& -1454.5\\
 \hline
Fast gradient (Nesterov)  &100&  54 & 0.029	&	1000& -1454.5\\
 \hline
Barzilai-Borwein long steps  &100 & 54  &$<10^{-16}$  &7 &-1454.5\\
 \hline
Barzilai-Borwein short steps   &100&54   & $<10^{-16}$& 7	& -1454.5\\
 \hline
 Conjugate gradient  &100&54 & $<10^{-16}$ 	& 2&-1454.5\\
 \hline
 \hline
 ME &150&   68  & $<10^{-16}$ &	1& -2319.5 \\
 \hline
 Gradient optimal step  &150 &68  &  0.0009& 13  	& -2319.5\\
 \hline
Fast gradient (Nesterov)  &150& 68  & 0.07	&	1000& -2319.5\\
 \hline
Barzilai-Borwein long steps  &150&68 & 0.0009 &7 &-2319.5\\
 \hline
Barzilai-Borwein short steps   &150&68   & $<10^{-16}$& 7	& -2319.5\\
 \hline
 Conjugate gradient  &150&68 & $<10^{-16}$ 	& 2&-2319.5\\
 \hline
 \end{tabular}
 \vspace{0.5cm}
\caption{Comparison on several instances  with dense matrix $A$ of the number of iterations, CPU time, and optimal value at termination for  ME,
gradient with optimal step, fast gradient, Barzilai-Borwein with long steps, 
Barzilai-Borwein with short steps, and conjugate gradient methods using $\varepsilon=1$.}\label{tableres3}
\end{table}
}}

{\tiny{
\begin{table}
\centering
\begin{tabular}{|c|c|c|c|c|c|}
 \hline
Method & n & $\displaystyle \frac{\lambda{\max}(A)}{\lambda_{\min}(A)}$ &  CPU time (s) & Iterations & Optimal value \\
 \hline
ME &200&  99& 0.0001  & 2  &-4901.9	\\
 \hline
 Gradient optimal step  & 200 &99 & 0.0009   & 11   	& -4901.9\\
  \hline
Fast gradient (Nesterov)  &200& 99  &	0.09 & 1000	& -4901.9\\
 \hline
Barzilai-Borwein long steps  & 200 &99& 0.0009 & 7 &	-4901.9 	\\
 \hline
Barzilai-Borwein short steps  &200& 99&$<10^{-16}$ & 7 	& -4901.9\\
 \hline
 Conjugate gradient  &200  & 99 &$<10^{-16}$ & 2	& -4901.9\\
\hline
\hline
ME & 300 & 138    & 0.001 &	2&-9534.5\\
 \hline
 Gradient optimal step  & 300& 138 &0.0009& 11  &-9534.5 \\
 \hline
Fast gradient (Nesterov)  &300&  138  &0.27& 	1000	& -9534.5\\
 \hline
Barzilai-Borwein long steps  &300& 138 &0.002& 7  &	-9534.5 	\\
 \hline
Barzilai-Borwein short steps  &300& 138& 0.001& 7 	& -9534.5\\
 \hline
 Conjugate gradient  &300& 138 & 0.001 	&2 &-9534.5\\
 \hline
 \hline
ME &400&  215   & 0.001  &2	&-23098\\
 \hline
 Gradient optimal step  &400 &215 &0.002 & 9  	&-23098 \\
 \hline
Fast gradient (Nesterov)  &400& 215 &0.27 &1000		& -23098\\
 \hline
Barzilai-Borwein long steps  &400& 215&0.003& 7 &	 -23098	\\
 \hline
Barzilai-Borwein short steps  &400&215 &0.004& 7 	& -23098\\
 \hline
 Conjugate gradient &400 &215&    0.001	& 2&-23098\\
 \hline
 \hline
 ME &500  & 241    & 0.002 &	2&-29290\\
 \hline
 Gradient optimal step  &500 & 241& 0.003&  11   	& -29290\\
 \hline
Fast gradient (Nesterov)  &500& 241 &0.38 & 1000		& -29290\\
 \hline
Barzilai-Borwein long steps  &500&241 & 0.007& 7&-29290 \\
 \hline
Barzilai-Borwein short steps  &500&241 &0.007  &7  	& -29290\\
 \hline
 Conjugate gradient  &500&241 &  0.001	& 2&-29290\\
 \hline
 \hline
ME &600  & 309    & 0.003 &2	&-47732\\
 \hline 
Gradient optimal step  &600 & 309 &  0.003 &9  	& -47732\\
 \hline
Fast gradient (Nesterov)  &600& 309  &0.6 &	1000	& -47732\\
 \hline
Barzilai-Borwein long steps  &600&309 & 0.01& 7& -47732\\
 \hline
Barzilai-Borwein short steps  &600&309 & 0.009 & 7 	& -47732\\
 \hline
 Conjugate gradient  &600&309 & 0.002 	& 1 &-47732\\
 \hline
 \hline
 ME &700  & 353    & 0.02 &2	&-62302\\
 \hline 
Gradient optimal step  &700 & 353 & 0.04  & 9 	& -62302\\
 \hline
Fast gradient (Nesterov)  &700&  353 &4.1 &1000		&-62302 \\
 \hline
Barzilai-Borwein long steps  &700& 353& 0.08 & 7 & -62302\\
 \hline
Barzilai-Borwein short steps  &700&353 & 0.07 & 7 	& -62302\\
 \hline
 Conjugate gradient  & 700 & 353 & 0.01 & 1 & -62302\\
 \hline
 \hline
  ME &1000  & 353    & 0.03 &	2&-114741\\
 \hline 
 Gradient optimal step  &1000 & 353 & 0.04 & 9  	&-114741 \\
 \hline
Fast gradient (Nesterov)  &1000&  353 &4.2 &	1000	& -114741\\
 \hline
Barzilai-Borwein long steps  &1000&353 & 0.08& 7& -114741\\
 \hline
Barzilai-Borwein short steps  &1000&353 & 0.24 & 7 	&-114741 \\
 \hline
 Conjugate gradient  &1000&353 & 0.02 	&2 &-114741\\
 \hline
 \end{tabular}
 \vspace{0.5cm}
\caption{Comparison on several instances (for several values of problem size $n$) with dense matrix $A$ of the number of iterations, CPU time (in seconds), and optimal value at termination for ME,
gradient with optimal step, fast gradient (Nesterov), Barzilai-Borwein with long steps, 
Barzilai-Borwein with short steps, and conjugate gradient methods using $\varepsilon=1$.}\label{tableres4}
\end{table}
}}

\section{Concluding remarks}\label{secconc}

In this paper, we introduced the Method of Ellipcenters (ME) to minimize a
differentiable function. We fully derived a framework for ME in the general
case and proposed a detailed algorithm in the particular case of the minimization
of a quadratic objective given by a definite positive matrix.
In this case, we proved the linear convergence of the method and obtained
encouraging numerical results showing that ME is competitive with state of the art
methods for minimizing quadratic functions.

In a future work, we intend to study more in depth the extension of the method
for the minimization of a convex function.
In the case when the function is strongly convex, it is easy to check
that $y^k$ is well defined and uniquely defined, as we show next, in Lemmas
\ref{uniqueness} and \ref{existence}. We need the following definition.
 
\begin{definition}[Line] A line $L$ is a set of form
$
L=\{tx+(1-t)y:t \in \mathbb{R}\}
$
where $x \neq y$ are two points in $\mathbb{R}^n$.
\end{definition}

\begin{lemma}\label{uniqueness} Let $f: \mathbb{R}^n \rightarrow \mathbb{R}$ be a $\mu-$strongly convex function for some norm $\|\cdot\|$. Given a line $L$, there
are at most two different points $x,y \in L$ where $f$ has the same value,
i.e., such that $f(x)=f(y)$.
\end{lemma}
\begin{proof} Assume by contradiction that there are three
different points $x, y, z$ in a line $L$ such that $f(x)=f(y)=f(z)$.
Denote by $\tilde f$ the common value of $f$ at these points
and without loss of generality assume that $y$ belongs
to the segment $[x,z]=\{tx+(1-t)z: 0 \leq t \leq 1\}$, i.e.,
$y=tx+(1-t)z$ for some $0<t<1$. Then
$$
\tilde f=f(y)=f(tx+(1-t)z) \leq 
tf(x)+(1-t)f(z)-\frac{\mu t(1-t)}{2}\|x-z\|^2
=\tilde f -\frac{\mu t(1-t)}{2}\|x-z\|^2
$$
which implies $x=z$ (since $\mu>0$ and
$1>t>0$) and yields the desired contradiction. 
\end{proof}

\begin{lemma}\label{existence}
Let $f: \mathbb{R}^n \rightarrow \mathbb{R}$ be a differentiable and $\mu-$strongly convex function for some norm $\|\cdot\|$. Let $x \in \mathbb{R}^n$ such that
$\nabla f(x) \neq 0$. Then
there is one and only one $t>0$
such that
$f(x-t\nabla f(x))=f(x)$.
\end{lemma}
\begin{proof}
Let $g(t)=f(x-t\nabla f(x))$.
Since $g'(0)<0$, there is
$t_0>0$ such that
$g(t_0)<g(0)$. Since $f$
is strongly convex on $\mathbb{R}^n$, it is continuous
on $\mathbb{R}^n$ and coercive, i.e., $\lim_{t \rightarrow +\infty} g(t)=+\infty$, implying
that there is $t$ satisfying
$t_0<t<+\infty$ such that
$g(t)=g(0)$ or equivalently such that $f(x-t\nabla f(x))=f(x)$. By Lemma \ref{uniqueness}, there cannot be more than one such $t>t_0$, which achieves the proof of the lemma.
\end{proof}

Extension of the method for strongly convex problems
involves writing conditions ME1, ME2, and ME3
and choosing among ellipses satisfying ME1, ME2, ME3
one whose center is, for instance, closest to $x^k$.
In this more general case, $y^k$ and step $t_k$
can be computed by dichotomy, based on Lemmas
\ref{uniqueness} and 
\ref{existence}. It would also be interesting to study the complexity
of ME.\\

\par {\textbf{Funding.}} Roger Behling was partially supported by Conselho Nacional de Desenvolvimento Científico e Tecnológico (CNPq) Grant 309458/2025-0 and Fundacao de Amparo a Pesquisa e Inovacao do Estado de Santa Catarina (FAPESC) Grant 2024TR002238. Vincent Guigues was partially supported by CNPq acknowledges the support of CNPq grant 305263/2023-4. \\

\par {\textbf{Competing interests.}} The authors have no competing interests to declare.

\end{document}

%% file: def.tex
\usepackage{amsmath}
\usepackage{amsfonts}
\usepackage{latexsym}
\usepackage{amssymb}

\newtheorem{thm}{Theorem}[section]
\newtheorem{theorem}[thm]{Theorem}

\newtheorem{lemma}[thm]{Lemma}

\newtheorem{proposition}[thm]{Proposition}

\newtheorem{definition}[thm]{Definition}

\newtheorem{remark}[thm]{Remark}

\newcommand{\beq}{\begin{equation}}
\newcommand{\eeq}{\end{equation}}
\newcommand{\beqa}{\begin{eqnarray}}
\newcommand{\eeqa}{\end{eqnarray}}
\newcommand{\beqas}{\begin{eqnarray*}}
\newcommand{\eeqas}{\end{eqnarray*}}
\newcommand{\bi}{\begin{itemize}}
\newcommand{\ei}{\end{itemize}}

\newcommand{\nn}{\nonumber}